\documentclass[english,a4paper]{amsart}
\usepackage{babel}

\usepackage{enumerate}

\usepackage[T1]{fontenc}
\usepackage[utf8]{inputenc}

\usepackage{amsmath}
\usepackage{amssymb}
\usepackage{amsthm}
\usepackage{xypic}

\theoremstyle{plain}
\newtheorem{theorem}{Theorem}
\newtheorem*{theorem*}{Theorem}
\newtheorem{proposition}[theorem]{Proposition}
\newtheorem{corollary}[theorem]{Corollary}
\newtheorem*{corollary*}{Corollary}
\newtheorem{lemma}[theorem]{Lemma}

\theoremstyle{remark}
\newtheorem{remark}[theorem]{Remark}
\newtheorem{example}[theorem]{Example}

\theoremstyle{definition}
\newtheorem{definition}[theorem]{Definition}

\newcommand{\Zset}{\mathbb{Z}}
\newcommand{\Cset}{\mathbb{C}}
\newcommand{\field}[1][]{\mathbb{F}_{#1}}

\newcommand{\rann}[2]{\operatorname{rann}_{#1}\left( #2 \right)}
\newcommand{\lann}[2]{\operatorname{lann}_{#1}\left( #2 \right)}
\newcommand{\lorth}[1]{{}'#1}
\newcommand{\rorth}[1]{#1'}

\newcommand{\bilin}[2][]{\left\langle #2 \right\rangle_{#1}}

\newcommand{\lt}[1]{\operatorname{lt}(#1)}
\newcommand{\Soc}{\operatorname{Soc}}
\newcommand{\card}[1]{|#1|}

\begin{document}

\title[Frobenius algebras and linear codes]{Some remarks on non projective Frobenius algebras and linear codes}

\thanks{Research partially supported by grants {MTM2013-41992-P} from {MINECO} and {MTM2016-78364-P} from {Agencia Estatal de Investigaci\'{o}n (AEI) and from Fondo Europeo de Desarrollo Regional (FEDER)}}

\author[G\'omez et. al.]{Jos\'e G\'omez-Torrecillas}
\address{CITIC and Department of Algebra, University of Granada}
\email{gomezj@ugr.es}

\author[]{Erik Hieta-aho}
\address{Department of Mathematics and Computer Science, University of Puget Sound}
\email{ehietaaho@pugetsound.edu}

\author[]{F. J. Lobillo}
\address{CITIC and Department of Algebra, University of Granada}
\email{jlobillo@ugr.es}

\author[]{Sergio L\'{o}pez-Permouth} 
\address{Department of Mathematics, Ohio University}
\email{lopez@ohio.edu}

\author[]{Gabriel Navarro}
\address{CITIC and Department of Computer Science and AI, University of Granada}
\email{gnavarro@ugr.es}


\maketitle

\begin{abstract}
With a small suitable modification, dropping the projectivity condition, we extend the notion of a Frobenius algebra to grant that a Frobenius algebra over a Frobenius commutative ring is itself a Frobenius ring. The modification introduced here also allows Frobenius finite rings to be precisely those rings which are Frobenius finite algebras over their characteristic subrings. From the perspective of linear codes, our work expands one's options to construct new finite Frobenius rings from old ones. We close with a discussion of generalized versions of the MacWilliams identities that may be obtained in this context. 
\end{abstract}

\section{Introduction}

The core ingredient of the definition of a Frobenius algebra $A$ finitely generated as a module over its commutative ground ring $K$, as it appears in the literature, is the requirement that there exist a non degenerate associative $K$--bilinear form on  $A$; this requirement mimics accurately the definition of a Frobenius algebra over a field. However, that definition also includes the somewhat technical requirement that $A$ be projective as a $K$--module. At first, since all modules over a field are indeed free and hence projective, the extra condition appears innocuous. Furthermore, the extra condition apparently served the purpose of affording the result, mentioned without a detailed proof in \cite[p. 434]{Lam:1999}, that every Frobenius algebra over a Frobenius commutative ring is a Frobenius ring. Unfortunately, not every finite Frobenius ring is projective over its characteristic subring. 

In this paper we remove the projectivity technical requirement from the usual definition of a Frobenius algebra over a commutative ring and propose that said expression be used instead for a  properly modified notion. We show that thus extending the definition affords us a very natural result, namely, that  a finite ring $A$ of characteristic $n$ is Frobenius if and only if it is a (not necessarily projective) Frobenius  algebra over its characteristic ring $\Zset_n$. We, therefore, adopt the expression non projective Frobenius algebra over a ring when all else is satisfied but possibly not the projectivity requirement. Notice that, as is a common practice in mathematics, our expression “non projective”, is a less pedantic option to express “not necessarily projective”.  

An additional benefit from our approach is that we may adapt, without using the projectivity of $A$ over $K$, the arguments in the usual proof that Frobenius algebras over a field are Frobenius as rings to the case of non projective Frobenius algebras over a ring  (Theorem \ref{FrobeniuseanFrobenius}). This result is interesting in its own right and has potential to open many doors in the realm of applications.   Note that one may also derive Theorem \ref{FrobeniuseanFrobenius} from the recent analysis of the Frobenius property for Artin algebras in \cite{Iovanov:2016}.  However, while more general, this approach is significantly more technical and we felt that there is value in sharing our proof here in hopes that it would be more accessible for a larger group of interested readers.  

One of the motivations behind this extension of the expression {\it non projective Frobenius algebra over a ring} arises because of applications of this notion in the algebraic theory of error-correcting codes. Finite Frobenius rings are extensively used as alphabets for ring-linear block codes;  Frobenius algebras over a finite field are an important example of Frobenius rings. In contrast, a well known characterization of Frobenius rings \cite{Wood:1999} may be rephrased by asserting that a finite ring is Frobenius if and only if there exists a non degenerate associative bilinear form over its characteristic subring.  

The fundamental problem in the theory of linear codes is the analysis and classification of linear codes over an alphabet, usually a finite field $F$.  Since there is a tremendous number of subspaces for the vector space $F^n$ the task at hand is quite monumental. A common frequently successful practice has been to filter the codes considered in the following manner: add an $F$-algebra structure on the vector space $F^n$ (let us refer to $F^n$, equipped with the additional structure, as $A_n$ to remind us of the fact that it is now an algebra of dimension $n$) and consider, say, only those subspaces which are (left, right, two-sided) ideals of $A_n$.  We refer to the additional algebraic structure on $F^n$ as the {\it ambient} (or the {\it ambient space}) of the codes targeted and to the criterion applied to choose the codes we wish to consider as the {\it filter} being applied. 

As the interest in studying codes in ambients endowed with additional algebraic structure expanded from considering field alphabets to ring alphabets, Frobenius rings have tended to become central to the conversation.  The reason for this preference goes back to papers that show that the fundamental structural properties of codes over fields extend to codes that have Frobenius rings as their alphabets.  In fact, there is plenty of evidence to show that Frobenius rings are precisely the extent to which this type of considerations can be extended if one wishes to acquire the desired structural properties. For the properties to hold, the alphabet must be Frobenius.

In this paper, we investigate an alternative role that Frobenius rings appear to play very naturally in this context.  It seems to be the case that coding-theoretic structural properties are attained, regardless of the structure of the alphabet, when the ambient is a non projective Frobenius algebra over a commutative ring possibly other than the alphabet.  For example, it may be that the ambient is a non projective Frobenius algebra over its characteristic subring.

We illustrate our ideas by taking advantage of the fact that every finite Frobenius ring of characteristic $n$ can be seen as a non projective Frobenius $\Zset_n$--algebra. Our examples include constructing a family of  non projective Frobenius algebras based on factor algebras of skew polynomial algebras and using them as the ambient algebra for a  wide class of analogues to skew cyclic block codes.

This note is organized as follows. Section \ref{Anuladores} provides, with the generality needed here, the statements and proofs of  preliminary results on balanced (also called associative) bilinear forms over bimodules; we only include proofs that we deemed illustrative. Section \ref{sec:Frobeniusean} provides our proposed definition of a non projective Frobenius algebra over a commutative ring and various equivalent characterizations; as mentioned above, projectivity of the algebra over the base ring is not required.  One of these equivalent conditions is the existence of a Frobenius functional, which will be seen to generalize the generating character and play a similar role to it. In Section \ref{FrobeniuseanAnn}, we extend  results from \cite{Szabo/Wood:2017} on annihilators associated to a non degenerate bilinear form from a finite Frobenius ring to a non projective Frobenius algebra. Section \ref{FF} contains our observation that, given an algebra $R$ over a Frobenius commutative ring $K$ such that $R$ is finitely generated as a $K$--module, then $R$ is a non projective Frobenius algebra over $K$ if and only if $R$ is a Frobenius ring. This, in particular, applies to finite rings, viewed as algebras over their characteristic subrings. We also include a method, based on skew polynomials, to construct new Frobenius algebras from a given one. From the point of view of codes, this gives a way to construct new finite Frobenius rings from old ones; this is further discussed in Section \ref{Codes}. That section also contains a discussion of how the general results on bilinear forms defined on modules over non projective Frobenius algebras, developed in the previous sections, may be applied to get the main results of \cite{Szabo/Wood:2017}. We close that section with a discussion of those bilinear forms for which a version of MacWilliams identities stated in \cite{Szabo/Wood:2017} holds. 

\section{Preliminaries on bilinear forms and annihilators}\label{Anuladores}
All rings we consider will be unital and possibly non commutative. Let $A, B$ be rings and $M$ an $A-B$--bimodule. We will simply say that ${}_AM_B$ is a bimodule. We also use the notation ${}_AM$ to declare that $M$ is a left $A$--module, and $M_B$ for right $B$--modules. Systematic  introductions to the theory of modules over non commutative rings are \cite{Kasch:1982}, or \cite{Anderson/Fuller:1992}.
Everything in this section follows from \cite[\S 30]{Anderson/Fuller:1992}. For convenience, however, we state the results here in the form they are needed, and provide their proofs in this more limited scope to keep the paper self-contained. Given a second bimodule ${}_BN_{A}$, the abelian group $N^* = \hom_A(N,A)$ of all homomorphisms of right $A$--modules is endowed with the structure of a bimodule ${}_AN^*_{B}$ by the rule
\[
(afb)(n) = af(bn), \qquad a \in A, b \in B, f \in N^*, n \in N.
\]
A straightforward argument shows that the formula 
\[
\bilin{m,n} = \alpha(m)(n), \qquad m \in M, n \in N
\]
provides a bijective correspondence between $A$--bilinear maps 
\[ 
\bilin{-,-}: M \times N \to A
\]
 and homomorphisms of left $A$--modules 
 \[
 \alpha: M \to N^*.
 \]
  Moreover, $\bilin{-,-}$ is \emph{associative (or balanced)}, in the sense that $\bilin{mb,n} = \bilin{m,bn}$ for all $m \in M, n\in N, b \in B$, if and only if the corresponding $\alpha : M \to N^*$ is a homomorphism of right $B$--modules. 

Analogously, the additive group ${}^*M = \hom_A(M,A)$ of all homomorphisms of left $A$--modules is a $B-A$--bimodule via the rule
\[
(bfa)(m) = f(mb)a, \qquad b \in B, a \in A, f \in {}^*M. 
\]
The assignment
\[
\bilin{m,n} = \beta(n)(m), \qquad m \in M, n \in N
\]
gives a bijective correspondence between $A$--bilinear forms 
\[ \bilin{-,-} : M \times N \to A
\] 
and homomorphisms of right $A$--modules 
\[
\beta : N \to {}^*M. 
\]
Again, the condition of being the bilinear form associative is equivalent to require that $\beta$ is a homomorphism of left $B$--modules. 

\begin{definition}\label{ndegdef}
Let $\bilin{-,-} : M \times N \to A$ be an $A$--bilinear form, and $\alpha : M \to N^*$, $\beta : N \to {}^*M$ the corresponding homomorphisms of $A$--modules defined as before. We say that $\bilin{-,-}$ is right (resp. left) \emph{non degenerate} if $\alpha$ (resp. $\beta$) is injective. When $\bilin{-,-}$ is left and right non degenerate, we just say that the bilinear form is non degenerate. 
\end{definition}

A left module ${}_AL$ is an \emph{injective cogenerator} if the exactness of a sequence of morphisms of left $A$--modules 
\begin{equation*}
\xymatrix{X \ar[r] & Y \ar[r] & Z}
\end{equation*}
is equivalent to the exactness of the corresponding sequence of additive groups
\begin{equation*}
\xymatrix{\hom_A(Z,L) \ar[r] & \hom_A(Y,L) \ar[r] & \hom_A(X,L)}.
\end{equation*}
That is, ${}_AL$ is both injective and cogenerator (see  \cite[Proposition 18.14]{Anderson/Fuller:1992}). 

Recall that  a left artinian ring $A$ is Quasi-Frobenius (QF) if ${}_AA$ is an injective cogenerator. Equivalently, $A$ is right artinian and $A_A$ is an injective cogenerator (see \cite[Theorem 13.2.1]{Kasch:1982}).

The length of a right $A$--module $X$ will be denoted by $\lt{X_A}$, for a left $A$--module $Y$, by $\lt{{}_AY}$.  From now on, unless otherwise stated, we will assume that all $A$--modules are of finite length. A (left and right) artinian  ring $A$  is QF if and only if  $\lt{{}_AX^*} = \lt{X_A}$ for every finitely generated right $A$--module $X_A$ and $\lt{{}^*Y_A} = \lt{{}_AY}$ for every finitely generated left $A$--module ${}_AY$ (see \cite[Theorem 13.3.2]{Kasch:1982}).   We assume in the rest of this section that $A$ is a Quasi-Frobenius ring (see \cite[Ch. 13]{Kasch:1982} for various additional characterizations of these rings, including the original definition due to Nakayama).

\begin{lemma}\label{lmln}
If there is a non degenerate $A$--bilinear form $\bilin{-,-} : M \times N \to A$, then $\lt{{}_AM} = \lt{N_A}$.
\end{lemma}
\begin{proof}
The following computation, which uses that both $\alpha$ and $\beta$ are injective maps, gives the statement. 
\[ 
\lt{{}_AM} \leq \lt{{}_AN^*} = \lt{N_A} \leq \lt{{}^*M_A} = \lt{{}_AM}.
\]
\end{proof}

Next lemma, as well as Proposition \ref{dobledual} below, can be deduced from \cite[Theorem 30.1]{Anderson/Fuller:1992}. 

\begin{lemma}\label{ndeg}
Assume that  $\lt{{}_AM} = \lt{N_A}$. An $A$--bilinear form $\bilin{-,-} : M \times N \to A$ is left non degenerate if and only if it is right non degenerate. In such a case, both $\alpha$ and $\beta$ are isomorphisms. 
\end{lemma}
\begin{proof}
Assume that $\bilin{-,-}$ is right non degenerate, that is, the homomorphism of left $A$--modules $\alpha : M \to N^*$ is injective. Since $A$ is QF, $\lt{{}_AN^{*}} = \lt{N_A} = \lt{{}_AM}$. Therefore, $\alpha$ is an isomorphism. In order to prove that $\bilin{-,-}$ is left non degenerate, let $n \in N$ such that $\beta(n) = 0$. This means that, for every $m \in M$, $0 = \beta(n)(m) = \bilin{m,n} = \alpha(m)(n)$. Since $\alpha$ is surjective, we get that $\varphi (n) = 0$ for all $\varphi \in N^*$. This implies that $n = 0$, since $A$ is an injective cogenerator right $A$--module. Therefore, $\beta$ is injective. We so far have proved that if $\bilin{-,-}$ is right non degenerate, then it is left non degenerate, and $\alpha$ is an isomorphism. By symmetry, we get the full statement of the lemma. 
\end{proof}

Now, for  an $A$--bilinear form $\bilin{-,-} : M \times N \to A$, and subsets $S \subseteq N$, $T \subseteq M$, we define 
\begin{align*}
\lorth{S} &= \left\{ m \in M ~|~ \bilin{m,s} = 0, \,\forall s \in S \right\}, \\
\rorth{T} &= \left\{ n \in N ~|~ \bilin{t,n} = 0, \,\forall t \in T \right\}.
\end{align*}
Clearly, $\lorth{S}$ is a left $A$--submodule of $M$, while $\rorth{T}$ is a right $A$--submodule of $N$. 

\begin{proposition}\label{dobledual}
Assume that $\bilin{-,-}$ is non degenerate. For every $A$--submodule $X$ (resp. $Y$) of $N_A$  (resp. of ${}_AM$) we have that $\rorth{(\lorth{X})} = X$ and $\lorth{(\rorth{Y})} = Y$. 
\end{proposition}
\begin{proof}
From Lemma \ref{lmln}, $\lt{{}_AM} = \lt{N_A}$. 
Consider the commutative diagram of left $A$--modules with exact rows
\[
\xymatrix{0 \ar[r] & (N/X)^* \ar[r] & N^* \ar[r] & X^* \ar[r] &  0 \\
0 \ar[r]  & \lorth{X} \ar[r] \ar^{\alpha''}[u] & M \ar[r] \ar^{\alpha}[u] & M/\lorth{X} \ar[r] \ar^{\alpha'}[u] & 0},
\]
where $\alpha'(m + \lorth{X})(n) = \bilin{m,n}$ for all $m \in M, n \in X$, and $\alpha''(m)(n + X) = \bilin{m,n}$ for all $m \in \lorth{X}, n \in N$.  We see that $\alpha'$ is injective, so,  $\alpha$ being an isomorphism by Lemma \ref{ndeg}, we get that $\alpha'$ is an isomorphism, too.  This implies that $\alpha''$  is an isomorphism.  Therefore, $\lorth{X}\cong (N/X)^*$ and $X^* \cong M/\lorth{X}$ as $A$--modules. Analogously, we have isomorphisms of $A$--modules $\rorth{Y} \cong {}^*(M/Y)$ and ${}^*Y \cong N/\rorth{Y}$. 

Now, observe that $X \subseteq \rorth{(\lorth{X})}$. Thus, from the isomorphisms $\rorth{(\lorth{X})} \cong {}^*(M/\lorth{X})$ and $X^* \cong M/\lorth{X}$, we get
\[
\lt{\rorth{(\lorth{X})}} = \lt{{}^*(M/\lorth{X})}  = \lt{M/\lorth{X}} = \lt{X^*} = \lt{X}, 
\]
which is only possible if $X = \rorth{(\lorth{X})}$. Analogously, $\lorth{(\rorth{Y})} = Y$. 
\end{proof}

\begin{remark}\label{dualorto}
With the notation of Proposition \ref{dobledual}, we get from its proof that there are isomorphisms of $A$--modules $M/\lorth{X} \cong X^*$ and $N/\rorth{Y} \cong {}^*Y$,  as well as $\lorth{X} \cong (N/X)^*$ and $\rorth{Y} \cong {}^*(M/Y)$. Moreover,
\begin{equation}\label{N*:X}
\alpha(\lorth{X}) = \{\varphi \in N^* : \varphi(X) = 0 \}.
\end{equation}
\end{remark}

\begin{corollary}\label{antireticulo}
The map $\lorth{(-)}$ gives an anti-isomorphism, with inverse $\rorth{(-)}$, between the lattices of right $A$--submodules of $N$ and of left $A$--submodules of $M$. 
\end{corollary}

If $\bilin{-,-}$ is associative, then we easily get from Proposition \ref{dobledual} the following refinement of Corollary \ref{antireticulo}. 

\begin{corollary}\label{biantireticulo}
If $\bilin{-,-}$ is associative, then the map $\lorth{(-)}$ gives an anti-isomorphism, with inverse $\rorth{(-)}$, between the lattices of $B-A$--subbimodules of $N$ and of $A-B$--subbimodules of $M$. 
\end{corollary}

\begin{remark}
As already observed before, the statement of Proposition \ref{dobledual} can be deduced from \cite[Theorem 30.1]{Anderson/Fuller:1992}. Indeed, the key ingredient to get a double annihilator property with respect to a non degenerate bilinear form with values in a bimodule $U$ is that the duals with respect to $U$ take simples to simples (see \cite[Theorem 30.1]{Anderson/Fuller:1992} for the precise statement). This is the case for several formulations of the double annihilator property in the realm of linear codes with finite (bi)modules as alphabets, like \cite{Greferath/alt:2004}, \cite{Gluesing/Pllaha:2018}, \cite{Wood:1999}.
\end{remark}

\section{Non projective Frobenius algebras}\label{sec:Frobeniusean}

A left and right artinian ring $A$ with Jacobson radical $J$ is \emph{Frobenius} if there are isomorphisms of modules
\begin{equation}\label{AnilloFrobenius}
Soc({}_AA) \cong A/J, \qquad Soc(A_A) \cong A/J, 
\end{equation}
where the notation $Soc(X)$ stands for the socle of a (left or right) module $X$. Every Frobenius ring is QF \cite[Corollary 13.4.3]{Kasch:1982}.  Moreover, in order to prove that a given QF ring $A$ is Frobenius one only needs to check one of the isomorphisms in \eqref{AnilloFrobenius}. Every commutative QF ring is Frobenius \cite[p. 361]{Kasch:1982}.

Let $R$ be an algebra over a commutative Frobenius ring $K$. Setting $A = K$ and $B = R$ in the framework of Section \ref{Anuladores}, we may thus consider the bimodules  $N= {}_RR_K$ and $M= {}_KR_R$. Therefore, $R^* = {}^*R = \hom_K(R,K)$ is endowed with the right and left $R$--module structures
\[
(fb)(b') = f(bb') = (b'f)(b), \qquad b,b' \in R, f \in R^*.
\]
This gives, indeed, an $R-R$--bimodule structure on $R^*$. Next, we materialize the discussion of Section \ref{Anuladores} to this framework.
Recall that a $K$--bilinear form $\bilin{-,-} : R \times R \to K$ is associative if $\bilin{bb',b''} = \bilin{b,b'b''}$ for all $b, b', b'' \in R$. 

\begin{proposition}\label{Frobeniusean}
Let $R$ be an algebra over a commutative Frobenius ring $K$. Assume $R$ to be finitely generated as a $K$--module. The following structures related to $R$ are in bijective correspondence.
\begin{enumerate}
\item\label{uno} Associative non degenerate $K$--bilinear forms 
\[ \bilin{-,-} : R \times R \to K.
\]
\item\label{dos} Isomorphisms of right $R$--modules  $\alpha : R \to R^*$.
\item\label{tres} Isomorphisms of left $R$--modules $\beta : R \to R^*$. 
\item\label{cuatro} $K$--linear forms $\epsilon \in R^*$ such that $\epsilon R = R^*$. 
\item\label{cinco} $K$--linear forms $\epsilon \in R^*$ such that $R \epsilon = R^*$.
\end{enumerate}
\end{proposition}
\begin{proof}
The equivalence between \eqref{uno}, \eqref{dos} and \eqref{tres} follows from Lemma \ref{ndeg} and the discussion previous to Definition \ref{ndegdef}. We only discuss the equivalence between \eqref{dos} and \eqref{cuatro}, since that between \eqref{tres} and \eqref{cinco} is symmetric. If there is an isomorphism of right $R$--modules $\alpha : R \to R^*$, then $\epsilon = \alpha(1)$ generates $R^*$ as a right $R$--module. Conversely, given $\epsilon \in R^*$ such that $R^* = \epsilon R$, we have the surjective homomorphism of right $R$--modules $\alpha : R \to \epsilon R = R^*$ given by $\alpha(b) = \epsilon b$. Since $K$ is Frobenius, $\lt{R_K} = \lt{R^{*}_K}$, and  we get that $\alpha$ is an isomorphism, indeed. Let us finally argue that the linear form appearing in \eqref{cuatro} coincides with that of \eqref{cinco}. Let $b, b' \in R$.  Then $\bilin{b,b'} = \alpha(b)(b') = (\epsilon b)(b') = \epsilon(bb') = (b'\epsilon)(b) = \beta(b') (b)$.  Thus, $R \epsilon = R^*$. 
\end{proof}

\begin{definition}
An algebra $R$ over a  commutative Frobenius ring $K$  is said to be \emph{non projective Frobenius} if $R$ is finitely generated as a $K$--module and  there exits a non degenerate $K$--bilinear form $\bilin{-,-} : R \times R \to K$.  The $K$--linear form $\epsilon : R \to K$ provided by Proposition \ref{Frobeniusean} will be referred to as the \emph{Frobenius functional}. 
\end{definition}

\begin{remark}
It follows from Proposition \ref{Frobeniusean} that a Frobenius $K$--algebra in the sense of \cite{Eilenberg/Nakayama:1955} is non projective Frobenius. The converse is not true (see Remark  \ref{noproy}). 
\end{remark}

\begin{remark}
We have already seen that the additional structure maps that, according to Proposition \ref{Frobeniusean}, make a $K$--algebra non projective Frobenius, are related by the following equalities
\[
\alpha(b)(b') = \bilin{b,b'} = \beta(b')(b), 
\]
\[
\epsilon = \alpha(1) = \beta(1), \quad \alpha(b) = \epsilon b,  \quad \beta(b) = b \epsilon,
\]
\[
\epsilon (bb') = \bilin{b,b'},
\]
for all $b, b' \in R$. 
\end{remark}

\begin{remark}\label{sindies}
As a consequence of Lemma \ref{ndeg} and Proposition \ref{Frobeniusean}, and from their proofs, we see that, in the latter, we may replace condition \eqref{uno} by only requiring the bilinear form to be either left or right non degenerate. Also, in conditions \eqref{dos} and \eqref{tres} we only need to require one of $\alpha$ or $\beta$ to be surjective. 
\end{remark}

\begin{example}\label{ssF}
Every semisimple algebra $R$  over a commutative Frobenius ring $K$ with $R_K$ finitely generated is a non projective Frobenius algebra. By virtue of Wedderburn-Artin Theorem, it suffices by proving this when $R$ is simple. In this case, ${}_RR \cong \Sigma^n$ for a simple left $R$--module $\Sigma$. Now,  $R^* \cong (\Sigma^*)^n$ as right $R$--modules. Since $K$ is Frobenius, $R$ and $R^*$ have the same length as $K$--modules. Therefore, $\Sigma$ and $\Sigma^*$ have the same length as $K$--modules. This implies  (look at $R$ as a matrix ring over a division $K$--algebra) that $\Sigma^*$ is a simple right $R$--module. Since its multiplicity in $R^*$ is $n$, we deduce that $R \cong R^*$ as right $R$--modules. 
\end{example}

\section{Annihilators in non projective Frobenius extensions}\label{FrobeniuseanAnn}
Let us come back to the situation of a pair of bimodules ${}_AM_B$ and ${}_BN_A$ with a non degenerate $A$--bilinear map $\bilin{-,-} : M \times N \to A$. Assume, in addition, that $A$ is a non projective Frobenius algebra over a Frobenius commutative ring $K$, with Frobenius functional $\epsilon : A \to K$. The corresponding associative non degenerate $K$--bilinear form (see Proposition \ref{Frobeniusean}) will be denoted by $\bilin[\epsilon]{-,-}$, and it obeys the rule
\[
\bilin[\epsilon]{a,a'} = \epsilon(aa'), \quad a, a' \in A.
\]
Define, for each subset $S \subseteq N$, 
\[
{}^\epsilon S = \{ m \in M : \epsilon (\bilin{m,s}) = 0 \; \forall s \in S \},
\]
which is a $K$--submodule of $M$. Analogously, for every subset $T \subseteq M$ we get the $K$--submodule of $N$
\[
T^\epsilon  = \{ n \in N: \epsilon (\bilin{t,n}) = 0 \; \forall t \in T \}.
\]
Indeed, ${}^\epsilon S$ and $T^\epsilon$ are the left and right orthogonal $K$--submodules of $S$ and $T$ with respect to the $K$--bilinear form
\[
[-,-]_{\epsilon} : M \times N \to K, \qquad [m,n]_{\epsilon} = \epsilon(\bilin{m,n}).  
\]
Obviously, $\lorth{S} \subseteq {}^\epsilon S$ and $\rorth{T} \subseteq T^\epsilon$. 

\begin{proposition}\label{ortogonales}
Let $S$ and $T$ be $A$--submodules of $N_A$ and ${}_AM$, respectively.  Then
\[
\lorth{S} = {}^\epsilon S, \qquad \rorth{T} = T^\epsilon.
\]
Thus, ${}^\epsilon S$ and $T^\epsilon$ are $A$--submodules of, respectively, ${}_AM$ and $N_A$. Moreover, 
\[
\rorth{(\lorth{S})} = S = ({}^\epsilon S)^\epsilon, \quad \lorth{(\rorth{T})} = T = {}^\epsilon(T^\epsilon). 
\]
\end{proposition}
\begin{proof}
To prove the equality ${}^\epsilon S = \lorth{S}$, we just need to check the inclusion ${}^\epsilon S \subseteq \lorth{S}$.
Let $a \in A, m \in M, s \in S$. If $m \in {}^{\epsilon}S$, then, since $S$ is a submodule of $N_A$, we have
\[
0 = \epsilon(\bilin{m,sa}) = \epsilon (\bilin{m,s}a) = \bilin[\epsilon]{\bilin{m,s},a}.
\]
Now $a \in A $ is arbitrary and $\bilin[\epsilon]{-,-}$ is non degenerate, we thus get that $\bilin{m,s} = 0$. Hence, $m \in \lorth{S}$. For the proof of the second part, observe that the bilinear form $[-,-]_\epsilon$  is non degenerate: if $m \in M$ is such that $[m,n]_\epsilon = 0$ for all $n \in N$, then $m \in {}^\epsilon N = \lorth{N} = \{ 0 \}$. Hence, $[-,-]_\epsilon$ is left non degenerate. The argument on the right is similar. By Proposition \ref{dobledual}, $({}^\epsilon S)^\epsilon = S$ and ${}^{\epsilon}(T^\epsilon) = T$. 
\end{proof}

For a subset $S \subseteq A$, we have the left (resp. right) ideal of $A$ defined by
\[
\lann{A}{S} = \{ a \in A : as = 0 \; \forall s \in S \}
\]
(resp. by)
\[
\rann{A}{S} = \{ a \in A : sa = 0 \; \forall s \in S \}.
\]

\begin{corollary}\label{anuladores}
For any non projective Frobenius algebra $A$ with Frobenius functional $\epsilon : A \to K$, every right ideal $S$ of $A$, and every left ideal $T$ of $A$, we have
\[
\lann{A}{S} = {}^\epsilon S, \qquad \rann{A}{T} = T^\epsilon.
\]
Therefore,
\[
\rann{A}{\lann{A}{S}} = S, \qquad \lann{A}{\rann{A}{T}} = T.
\]
\end{corollary}
\begin{proof}
Setting $M = N = A$ in Proposition \ref{ortogonales}, and the bilinear form $\bilin{-,-}$ to be the multiplication map of $A$, we get the statement, since $\lorth{S} = \lann{A}{S}$, while $\rorth{T} = \rann{A}{T}$. 
\end{proof}

\begin{remark}
It follows from Corollary \ref{anuladores} and \cite[Theorem 13.2.1]{Kasch:1982} that every non projective Frobenius algebra over a Frobenius commutative ring is a Quasi-Frobenius ring. Next section is devoted to sharpen this result. 
\end{remark}

\section{Non projective Frobenius algebras and Frobenius rings}\label{FF}

Let $R$ an algebra over a commutative Frobenius ring $K$ with $R_K$ finitely generated. Let $J$ denote the Jacobson radical of $R$. Since $R$ is an artinian ring, the socle $Soc(X)$ of every $R$--module $X$  is an \emph{essential submodule},  in the sense that every nonzero submodule of $X$ intersects non trivially $Soc(X)$. It follows from \cite[Theorem 18.10 and Proposition 18.12]{Anderson/Fuller:1992} that two injective $R$--modules are isomorphic if, and only if, they have isomorphic socles.

It is well known that, if $X_R$ is finitely generated, then 
\[
\Soc(X_R) = \{ x \in X : xJ = 0\}.
\]
 Therefore, $\Soc(R^{*}_R)$ may be  computed as
\begin{equation}\label{socjac}
\Soc(R^{*}_R) = \{ f \in R^*: f(J) = 0 \} \cong (R/J)^{*}_R \cong (R/J)_R,
\end{equation}
where the last isomorphism of $R/J$--modules  (and, hence, of $R$--modules) holds because $R/J$ is semisimple, and Example \ref{ssF} applies.

\begin{theorem}\label{FrobeniuseanFrobenius}
Let $R$ be an algebra over a commutative Frobenius ring $K$ with $R_K$ finitely generated. Then $R$ is a non projective Frobenius $K$--algebra if and only if $R$ is a Frobenius ring.
\end{theorem}
\begin{proof}
Assume first that $R$ is Frobenius as a ring. Then $\Soc(R_R) \cong (R/J)_R$. By \eqref{socjac}, $(R/J)_R \cong \Soc(R^{*}_R)$. Now, $R_R$ is injective because it is Frobenius, and $R^{*}_R$ is injective because $K$ is Frobenius. Since they have isomorphic socles, it follows that $R_R \cong R^{*}_R$, and $R$ is a non projective Frobenius algebra. Conversely, Corollary \ref{anuladores} implies that $R$ is Quasi-Frobenius. Since $R$ is non projective Frobenius, we have an isomorphism of $R$--modules $R_R \cong R^{*}_R$. Therefore, taking \eqref{socjac} into account, we get
\[
\Soc(R_R) \cong \Soc(R^{*}_R) \cong (R/J)_R.
\]
Hence, $R$ is a Frobenius ring. 
\end{proof}

\begin{remark}
Theorem \ref{FrobeniuseanFrobenius} can also be deduced from \cite[Proposition 1.5]{Iovanov:2016}, since the commutative Frobenius base ring $K$ is a minimal injective cogenerator of the category of $K$--modules. 
\end{remark}

\begin{corollary}\label{Wood}
Let $R$ be a finite ring of characteristic $n$. Then $R$ is a Frobenius ring if and only if $R$ is a non projective Frobenius $\Zset_n$--algebra.
\end{corollary}

\begin{remark}\label{noproy}
A non projective Frobenius algebra need not be projective over its commutative base Frobenius ring. One of the simplest examples is the ring $R = \Zset_2 \times \Zset_4$, which is clearly Frobenius and, by Corollary \ref{Wood}, a non projective Frobenius algebra over $\Zset_4$. However, $R$ is not projective as a $\Zset_4$--module, so it is not a Frobenius algebra in the classical setting. 
\end{remark}

Next, we will describe a method for constructing Frobenius rings from skew polynomial rings with coefficients in a non projective Frobenius algebra. 

Let \(A\) be a non projective Frobenius  algebra over a Frobenius commutative ring $K$ with Frobenius functional $\epsilon : A \to K$, and associative non degenerate $K$--bilinear form
\[
\bilin[\epsilon]{-,-} : A \times A \to K, \quad \bilin[\epsilon]{a,b} = \epsilon(ab).
\]
Consider the skew polynomial ring  \(S = A[x;\sigma]\), where $\sigma$ is a $K$--algebra automorphism of $A$, and let  \(f = \sum_{i=0}^m f_i x^i \in S$ be a monic polynomial such that $Sf = fS$.  Since $Sf$ is a twosided ideal of $S$, we get the $K$--algebra  \(\mathcal{R} = S/Sf\), which is finitely generated as a $K$--module because $f$ is monic.  Indeed, every $g \in S$ can be written as $g = qf + r$, for suitable $q, r \in S$ with $r$ of degree smaller than $m$. This implies, always with $f$ monic, that ${}_A\mathcal{R} \cong A^m$. In other words, we will identify the elements of $\mathcal{R}$ with polynomials in $S$ with degree less than $m$, with the operations made modulo $f$. For $g \in \mathcal{R}$, the notation $g_0$ stands for its term of degree $0$. Finally, we assume that $f_0$ is a unit of $A$. 

\begin{theorem}\label{Frobeniusskew}
\(\mathcal{R} \) is a non projective Frobenius \(K\)-algebra with nondegenerate associative bilinear form
\[
\bilin{-,-} : \mathcal{R} \times \mathcal{R} \to K, \qquad \bilin{g,h} = \epsilon((gh)_0).
\]
Therefore, $\mathcal{R}$ is a Frobenius ring. 
\end{theorem}

\begin{proof}
A straightforward computation modulo $f$ shows that, for $$g = \sum_{i=0}^{m-1}g_ix^i, h =  \sum_{i=0}^{m-1}h_ix^i \in \mathcal{R},$$ we have
\begin{equation}\label{bilformula}
(gh)_0 = g_0h_0 - \sum_{i=1}^{m-1}g_{m-i}\sigma^{m-i}(h_i)f_0. 
\end{equation}
Obviously, $\bilin{-,-}$ is $K$--bilinear. In order to prove that it is non degenerate it suffices, by Remark \ref{sindies}, to show that $\bilin{-,-}$ is right non degenerate. So, let $g \in \mathcal{R}$ such that $\bilin{g,h} = 0$ for all $h \in \mathcal{R}$. Let $\lambda \in A$. Taking $h = \lambda \in A$ in \eqref{bilformula}, we get that
\[
0 = \epsilon(g_0\lambda) = \bilin[\epsilon]{g_0,\lambda}.
\]
Since $\bilin[\epsilon]{-,-}$ is non degenerate, we deduce that $g_0 =0$. Now, let $i\in \{1, \dots, m\}$. Setting $h = \lambda x^i$ in \eqref{bilformula}, we have
\[
0 = -\epsilon(g_{m-i}\sigma^{m-i}(\lambda) f_0) = -\bilin[\epsilon]{g_{m-i}, \sigma^{m-i}(\lambda) f_0}.
\]
Now, $\lambda \in A$ is arbitrary, $f_0 \in A$ is a unit, and $\sigma^{m-i}$ is an automorphism of $A$, so, we get from the non degeneracy of $\bilin[\epsilon]{-,-}$ that $g_{m-i} = 0$. Thus, $g = 0$, and $\bilin{-,-}$ is non degenerate. By Theorem \ref{FrobeniuseanFrobenius}, $\mathcal{R}$ is a Frobenius ring. 
\end{proof}

\section{Codes with a Frobenius alphabet}\label{Codes}

Let $A$ be a finite ring. Considering $A$ as an additive group, it is well known that it is isomorphic to the group $\widehat{A}$ of its complex characters (the group homomorphisms from $A$ to $\mathbb{C}^\times$). The ring structure of $A$ leads to an additional $A$--bimodule structure on $\widehat{A}$ (see \cite{Wood:1999}). It was observed in \cite{Wood:1999} that $A$ is a Frobenius ring precisely when $\widehat{A}$ is cyclic either as a left or as a right $A$--module. A module generator is then called generating character. Let us see how this result fits in the theory so far developed in this paper for non projective Frobenius algebras.

Let $n$ be the characteristic of $A$,  and consider $\Zset_n \subseteq \Cset^\times$ as the group of $n$--th roots of unity. Recall that the abelian group $\Zset_n$ has a unique structure of ring, and that it is Frobenius. Now, we have
\[
\widehat{A} = \hom_{\Zset}(A,\Cset^\times) = \hom_{\Zset}(A,\Zset_n) = \hom_{\Zset_n}(A,\Zset_n) = A^*. 
\]
Thus, the $A$--bimodule $A^*$ is, in this case,  nothing but the bimodule of characters of $A$. 

 Therefore,  in view of  Corollary \ref{Wood}, the equivalent conditions of  Proposition \ref{Frobeniusean} characterize when the finite ring $A$ is Frobenius. Indeed, we may explicitly state the following reformulation of some results from \cite{Wood:1999}. 

\begin{theorem}\cite[Theorems 3.10 and 4.3]{Wood:1999}
Let $A$ be a finite ring of characteristic $n$. The following structures related to $A$ are in bijective correspondence.
\begin{enumerate}
\item\label{unoZ} Associative non degenerate bilinear forms 
$A \times A \to \mathbb{Z}_n.$
\item\label{dosZ} Isomorphisms of right $A$--modules  $\alpha : A \to \widehat{A}$.
\item\label{tresZ} Isomorphisms of left $A$--modules $\beta : A \to \widehat{A}$. 
\item\label{cuatroZ} Characters $\epsilon \in \widehat{A}$ such that $\epsilon A = \widehat{A}$. 
\item\label{cincoZ} Characters $\epsilon \in \widehat{A}$ such that $A \epsilon = \widehat{A}$.
\end{enumerate}
Moreover, $A$ is a Frobenius ring if and only if any of these structures does exist. 
\end{theorem}

Let $A$ be a finite Frobenius ring of characteristic $n$, and Frobenius functional (or generating character) $\epsilon : A \to \Zset_n$. Let $\bilin{-,-} : M \times N \to A$ a non degenerate bilinear form, where ${}_AM$ and $N_A$ are finite $A$--modules. We know that $\lt{{}_AM} = \lt{N_A}$ and, by Lemma  \ref{ndeg}, ${}_AM \cong {}_AN^*$ and $N_A \cong M^{*}_A$. Let us see that this framework covers the module-theoretical setting considered in \cite{Szabo/Wood:2017}.

\begin{example}\label{sequilin}
Consider an anti-automorphism $\theta : A \to A$, and a left module ${}_AM$. We can consider the right $A$--module $N$ whose underlying additive group is $M$, with the right $A$--module structure defined by $ma = \theta^{-1}(a)m$ for all $a \in A, m \in M$. Then, as already observed in \cite[Remark 4.9]{Szabo/Wood:2017}, a non degenerate $A$--bilinear form $\bilin{-.-} : M \times N \to A$ is, precisely, a non degenerate sesquilinear form in the sense of \cite[\S 3]{Szabo/Wood:2017}. When $M$ has to be considered as the word ambient space for $A$--linear codes, a natural choice is to put $M = A^n$ with its canonical left $A$--module structure.  
\end{example}

\begin{example}\label{Abilin}
Of course, the same finite abelian group $M$ may support both a left $A$--module structure and a right $A$--module structure. This case, considered in \cite[\S 4]{Szabo/Wood:2017}, is clearly covered by our general formalism. Here, a natural choice is $M = A^n$ with its canonical $A$--bimodule structure. 
\end{example}

Next, we will see how the results in \cite{Szabo/Wood:2017} are derived from our general theory. So, fix a non degenerate $A$--bilinear form $\bilin{-,-} : M \times N \to A$, where $A$ is a finite Frobenius ring of characteristic $n$, and ${}_AM$, $N_A$ are finite modules. Let $\epsilon : A \to \Zset_n$ be a Frobenius functional (that is, a generating character). According to examples \ref{sequilin} and \ref{Abilin}, this setting covers all cases considered in \cite{Szabo/Wood:2017}.

\begin{proposition}\cite[Proposition 3.7 and Proposition 4.7]{Szabo/Wood:2017}
If $T \subseteq M$ is a left $A$--submodule and $S \subseteq N$ is a right $A$--submodule, then 
\[
\rorth{(\lorth{S})} = S = ({}^\epsilon S)^\epsilon, \quad \lorth{(\rorth{T})} = T = {}^\epsilon(T^\epsilon). 
\]
\end{proposition}
\begin{proof}
In view of Corollary \ref{Wood}, this proposition is a particular case of Proposition \ref{ortogonales}.
\end{proof}

 By $\card{X}$ we denote the cardinal of a finite set $X$.

\begin{theorem}\cite[Theorem 3.6 and Theorem 4.6]{Szabo/Wood:2017}\label{cardinales}
Let $S \subseteq N_A$ and $T \subseteq {}_AM$ be submodules. Then
$\card{S}\card{\lorth{S}} = \card{M} = \card{N}$ and $\card{T}\card{\rorth{T}} = \card{M} = \card{N}$. 
\end{theorem}
\begin{proof}
Consider the non degenerate $\Zset_n$--bilinear form $[-,-]_\epsilon : M \times N \to \Zset_n$ defined by $[m,n]_\epsilon = \epsilon(\bilin{m,n})$ for $m \in M, n \in N$. By Proposition \ref{ortogonales}, $\lorth{S} = {}^\epsilon S$, the latter being the left orthogonal of $S$ with respect to $[-,-]_\epsilon$. Now, by Remark \ref{dualorto}, we have an isomorphism of $\Zset_n$--modules $M/{}^\epsilon S \cong \hom_{\Zset_n}(S, \Zset_n)$. Therefore, $\card{M} = \card{{}^{\epsilon} S} \card{\hom_{\Zset_n}(S, \Zset_n)}= \card{{}^\epsilon S}\card{S}$, since $\hom_{\Zset_n}(S, \Zset_n)$ is nothing but the character group of $S$. Now, $M \cong \hom_{\Zset_n}(N, \Zset_n) \cong N$, which gives that $\card{M} = \card{N}$. 
\end{proof}

Recall that the Hamming weight $wt(x)$ of a vector $x \in A^n$ is defined by the number of nonzero components of $x$. Given an additive code $C \subseteq A^n$, the Hamming weight enumerator of $C$ is the complex polynomial in two variables $X,Y$
\[
W_C(X,Y) = \sum_{x \in C} X^{n-wt(x)}Y^{wt(x)}.
\]
 A general version of MacWilliams identity appears in \cite[Theorem 5.2]{Szabo/Wood:2017}. Unfortunately, it is not valid for every non degenerate bilinear form, as the following example shows.

\begin{example}
Set $A = \field[2]$, and let $\bilin{-,-} : \field[2]^2 \times \field[2]^2 \to \field[2]$ the non degenerate $\field[2]$--bilinear form defined by the non singular matrix 
\[
Q = \begin{pmatrix}  1 & 1 \\ 0 & 1 \end{pmatrix},
\]
that is, $\bilin{x,y} = x^t Q y$ for all $x, y \in \field[2]^2$. Let $C \subseteq \field[2]^2$ the linear code $C = \{(0,0), (1,0) \}$.  Its dual (with respect to $\bilin{-,-}$) is $\lorth{C} = \{(0,0), (1, 1) \}$. Therefore,
\[
W_C(X,Y) = X^2 + XY, \qquad W_{\lorth{C}}(X,Y) = X^2 + Y^2.
\]
If the MacWilliams identity stated in \cite[Theorem 5.2]{Szabo/Wood:2017} were applicable to $\bilin{-,-}$, then we would have
\[
W_{\lorth{C}}(X,Y) = \frac{1}{2}W_C(X+Y,X-Y) = X^2 + XY.
\]
Therefore, the identities from \cite[Theorem 5.2]{Szabo/Wood:2017} do not hold for a general non degenerate bilinear form. 
\end{example}

Theorem 5.2 in \cite{Szabo/Wood:2017} is a consequence of \cite[Theorem 11.3]{Wood:2009} whenever the isomorphisms
\[
\alpha : A^m \to \widehat{A^m}, \qquad \alpha(x)(y) = \epsilon(\bilin{x,y})
\]
and
\[
\beta : A^m \to \widehat{A^m}, \qquad \beta(x)(y) = \epsilon(\bilin{y,x})
\]
are isometries with respect to the Hamming weights in $A^m$ and $\widehat{A^m}$. This is the statement of \cite[Lemma 5.3]{Szabo/Wood:2017}, which does not hold for an arbitrary non degenerate $A$--bilinear form on $A^m$.  The argument is valid, however, for suitable bilinear forms. To be more precise, let $e_i$ denote, for $i= 1, \dots, m$,  the vector of $A^m$ whose only nonzero component is the $i$--th, which is $1$. The Hamming weight of $\varphi \in \widehat{A^m}$ is computed as the Hamming weight of its image under the group isomorphism
\[
\widehat{A^m} \to \widehat{A}^m, \qquad \varphi \mapsto (\varphi(e_1 \cdot), \dots, \varphi(e_m \cdot)),
\]
where $\varphi(e_i \cdot)(a) = \varphi(e_ia)$ for all $a \in A$ and $i = 1, \dots m$. That is, $wt(\varphi)$ is the number of indexes $i \in \{1, \dots, m\}$ such that $\varphi(e_i \cdot) \neq 0$. Recall that a square matrix with coefficients in a ring is said to be \emph{monomial} if each row and column contains exactly one nonzero entry and that nonzero entry is a unit.

\begin{lemma}\label{isometria}
The group isomorphism $\alpha : A^m \to \widehat{A^m}$ preserves the Hamming weight if and only if the matrix $Q = (\bilin{e_i, e_j})_{1 \leq i,j \leq m}$ is monomial. 
\end{lemma}
\begin{proof}
Given $a, b \in A$, and $i, j \in \{ 1, \dots, m \}$, we have 
\begin{equation}\label{epsilonab}
\alpha_{ae_i}(e_j \cdot)(b) = \epsilon (\bilin{ae_i,e_jb}) = \epsilon(a\bilin{e_i,e_j}b).
\end{equation}
If $\alpha$ preserves the Hamming weight, then $wt(\alpha_{e_i}) = 1$ for each $i \in \{1, \dots, m \}$. Thus, given $i$, there is a unique $j \in \{1, \dots, m \}$ such that $\alpha_{e_i}(e_j \cdot) \neq 0$. By \eqref{epsilonab}, $\bilin{e_i,e_j} \neq 0$. Furthermore, for $k \neq j$, we get from \eqref{epsilonab} that $\epsilon(\bilin{e_i,e_k}b) = 0$ for every $b \in A$ and, hence, $\bilin{e_i,e_k} = 0$. We have thus seen that the $i$-th row of $Q$ has only one nonzero entry, namely $\bilin{e_i,e_j}$. This implies, taking into account that $\bilin{-,-}$ is non degenerate, that $Q$ is a monomial matrix. 

Conversely, assume that $Q$ is a monomial matrix. Since $\alpha$ is already an isomorphism of additive groups, in order to prove that it preserves the Hamming weight, it will suffice if we check that $wt(\alpha(ae_i)) = 1$ for every $0 \neq a \in A$ and every $i = 1, \dots, m$. Given $i$, let $j \in \{ 1, \dots, m \}$ the unique index such that $\bilin{e_i,e_j} \neq 0$. By \eqref{epsilonab}, $\alpha(ae_i)(e_k \cdot) = 0$ for all $k \neq j$. But $\alpha(ae_i)(e_j \cdot) \neq 0$ because, otherwise, $\epsilon(\bilin{ae_i,e_j}b) = 0$ for all $b \in A$, whence $\bilin{ae_i,e_j} = 0$. This is not possible, as $\bilin{-,-}$ is non degenerate  and $ae_i \neq 0$. 
\end{proof}

\begin{theorem}\cite[Theorem 5.2]{Szabo/Wood:2017}
Let $\bilin{-,-} : A^m \times A^m \to A$ be a non degenerate $A$--bilinear form such that the matrix $Q = (\bilin{e_i,e_j})_{1\leq i, j \leq m}$ is monomial. Let $C \subseteq A^m$be a left (resp. right) $A$--linear code $C \subseteq A^m$, and set $D = \rorth{C}$ (resp. $D = \lorth{C}$). Then
\[
W_D(X,Y) = \frac{1}{\card{C}}W_C(X + (\card{A} - 1)Y,X-Y).
\]
\end{theorem}
\begin{proof}
Under the hypotheses on $\bilin{-,-}$, $\alpha$ is an isometry by Lemma \ref{isometria}. By Proposition \ref{ortogonales}, $\lorth{C} = {}^\epsilon C$. Moreover, by \eqref{N*:X}, applied to the bilinear form $[-,-]_\epsilon = \epsilon \circ \bilin{-,-}$, $$\alpha(^\epsilon C) = \{ \varphi \in \widehat{A^m} : \varphi(C) = 0\}.$$ Thus, MacWilliams formula follows from \cite[Theorem 11.3]{Wood:2009} .   As for the $D=\rorth{C}$ case concerns, the same argument works, by virtue of an obvious version of Lemma \ref{isometria} for $\beta$. 
\end{proof}

We conclude this section by giving some examples of non projective Frobenius algebras for which the annihilators with respect to their associative bilinear forms are the Euclidean duals with respect to certain bases. The notation $X^\perp$ will be used to denote the Euclidean dual of a given code $X$, with respect to the Euclidean product which will be clear in each situation.

\begin{example}\label{groupcodes}
Let $G$ be a finite group, and consider the group algebra $A = \Zset_n G$, which is a basic example of Frobenius $\Zset_n$--algebra, with the associative non degenerate bilinear form defined by 
\[ \bilin{\sum_{g \in G}\alpha_gg,\sum_{h\in G}\beta_hh} = \sum_{g \in G}\alpha_g\beta_{g^{-1}}. \] Note that, if we denote by $[-,-]$ the obvious euclidean bilinear form on $A$, then  $[a,b] = \bilin{a,\theta(b)}$, for all $a, b \in A$, where $\theta :A \to A$ denotes the involution determined by $\theta(g) = g^{-1}$ for $g \in G$. Hence, $S^{\perp} = \theta(\rorth{S})$ for every subset $S$ of $A$, which implies that if $C$ is a left ideal of $A$, then $C^{\perp}$ is a left ideal of $A$, too. 
\end{example}

Let $\sigma$ be an automorphism of a finite Frobenius ring  $A$ of characteristic $n$. Consider a Frobenius functional (or generating character) $\epsilon : A \to \Zset_n$, and the non degenerate $\Zset_n$--bilinear form given by $\bilin{a,b} = \epsilon(ab)$ for all $a, b \in A$. Let $m$ be a multiple of the order of $\sigma$. The polynomial $x^m - 1 \in S = A[x;\sigma]$ is central, so we may consider the finite ring $\mathcal{A} = S/S(x^m-1)$, which is a Frobenius ring and a non projective Frobenius $\Zset_n$--algebra, according to Proposition \ref{Frobeniusskew}, with the nondegenerate associative $\Zset_n$--bilinear form 
\begin{equation}\label{biltraza}
\langle -,- \rangle_\mathcal{A} : \mathcal{A} \times \mathcal{A} \to \Zset_n, \qquad \langle f, g \rangle = \epsilon((fg)_0).
\end{equation}
On the other hand, $\mathcal{A}$ is a free $A$--module of rank $m$ with basis $\{1, x, \dots, x^{m-1}\}$, so we have the Euclidean $A$--bilinear form
\[
[-,-] : \mathcal{A} \times \mathcal{A} \to A, \qquad [f,g] = \sum_{i=0}^{m-1}f_ig_i.
\]

\begin{proposition}\label{sigmacyclic}
Let $\theta : \mathcal{A} \to \mathcal{A}$ be defined by 
\[
\theta \left(\sum_{i=0}^{m-1}f_ix^i\right) = \sum_{i=0}^{m-1}\sigma^{-i}(f_i)x^{-i}.
\] 
Then $\theta$ is a $\Zset_n$--algebra involution. Moreover, for every left $A$--submodule $V$ of $\mathcal{A}$, we have
\[
V^\perp = \lorth{\theta(V)}. 
\]
\end{proposition}
\begin{proof}
By \cite[Lemma 26]{Gomez/alt:2018}, $\theta$ is a $\Zset_n$--algebra involution. Now, for $f, g \in \mathcal{A}$, a straightforward computation gives that $(f\theta(g))_0 = [f,g]$. Therefore, by \eqref{biltraza}, 
\begin{equation}\label{FrobEu}
\langle f, \theta(g) \rangle_{\mathcal{A}} = \epsilon([f,g]). 
\end{equation}
Given $f \in V^\perp$, and $\theta(g) \in \theta(V)$, we get from \eqref{FrobEu} that $\langle f, \theta(g)\rangle_{\mathcal{A}} = 0$. Hence, $f \in \lorth{\theta(V)}$, and we obtain the inclusion of $\Zset_n$--modules $V^\perp \subseteq \lorth{\theta(V)}$. Now, since $\langle -,- \rangle_{\mathcal{A}}$ and $[-,-]$ are non degenerate, we may apply Theorem \ref{cardinales} to both bilinear forms, and get
\[
\card{\lorth{\theta (V)}} = \frac{\card{\mathcal{A}}}{\card{\theta(V)}} = \frac{\card{\mathcal{A}}}{\card{V}} = \card{V^{\perp}}
\]
which implies the equality $V^\perp = \lorth{\theta(V)}$. 
\end{proof}

Following \cite{Boucher/Ulmer:2009}, left  ideals of $\mathcal{A}$ are called $\sigma$--cyclic codes. The following consequence generalizes \cite[Corollary 18]{Boucher/Ulmer:2009} from fields to finite Frobenius rings. 

\begin{corollary}
If $C$ is a $\sigma$--cyclic code, then $C^\perp$ is a $\sigma$--cyclic code. 
\end{corollary}
\begin{proof}
Since $C$ is a left ideal of $\mathcal{A}$, we get that $\theta(C)$ is a right ideal of $\mathcal{A}$ and, therefore, $\lorth{\theta(C)}$ is a left ideal. By Proposition \ref{sigmacyclic}, $C^\perp$ becomes a left ideal of $\mathcal{A}$. 
\end{proof}

\textbf{Aknowledgement:} We thank both referees for their useful reports.

\end{document}